\documentclass[a4paper]{article}

\usepackage{amsmath, amsfonts, amsthm, mathrsfs, amssymb}
\usepackage{graphicx, subfigure, threeparttable, booktabs, tabularx}

\usepackage{bm}
\usepackage{hyperref}

\newtheorem{thm}{Theorem}[section]

\newtheorem{lem}{Lemma}[section]

\newtheorem{definition}{Definition}[section]
\newtheorem{remark}{Remark}[section]
\newtheorem{problem}{Problem}
\newtheorem{example}{Example}

\begin{document}

\title{Direct imaging for the moment tensor point sources of elastic waves}

\author{
Xianchao Wang\thanks{1. School of Astronautics, Harbin Institute of Technology, Harbin, China. \ 
2. Department of Mathematics, City University of Hong Kong, Kowloon, Hong Kong SAR, China. {\it xcwang90@gmail.com}},
Yukun Guo\thanks{School of Mathematics, Harbin Institute of Technology, Harbin, China. {\it ykguo@hit.edu.cn}(Corresponding author)}\ \ and 
Sara Bousba\thanks{Department of Mathematics, Ferhat Abbas University, Setif, Algeria. {\it bousbasara90@gmail.com}}
}
\date{}


\maketitle

\begin{abstract}
We investigate an inverse source problem of the time-harmonic elastic wave equation. Some novel sampling-type numerical schemes are proposed to identify the moment tensor point sources in the Lam\'e system  from near-field measurements. Rigorous theoretical justifications are provided to show that the locations and moment tensors of the elastic sources can be uniquely determined from the multi-frequency displacement data. Several numerical examples are also presented to illustrate the validity and robustness of the proposed method.
\\

\noindent{\bf Keywords:}~~Inverse source problem, elastic wave, moment tensor point source, direct sampling method.
\end{abstract}


\section{Introduction}

The inverse source problems (ISPs) are associated with the identifications of unknown sources from the measured radiating field. The ISPs have attracted considerable attention in scientific fields and engineering applications, such as antenna synthesis \cite{Angel1991,Devaney2007}, medical imaging \cite{Arridge1999,  Liu2015} and seismic monitoring \cite{Albanese2014, Song2011}.

Due to the existence of non-radiating sources, the time-harmonic inverse source problems for wave equations with a single frequency may possess non-unique solutions \cite{Bleistein1977, Albanese2006}. Roughly speaking, two categories of approaches have been proposed in the literature to tackle the difficulty of non-uniqueness.
The first approaches utilize a single fixed frequency but assume that some \textit{a priori} information on the sources is available. We refer to Kusiak and Sylvester \cite{Kusiak2003, Kusiak2005} and Griesmaier \textit{et al} \cite{Griesmaier2012, Griesmaier2013} for the relevant studies on these strategies. The other approaches resort to the multi-frequency measurement data, that is,  the underlying frequency varies within certain open interval. For some recent investigations on the multi-frequency inverse source problems, we refer the reader to the recursive algorithm \cite{Bao2015, Bao2015IP, Bao2017}, the eigenfunction expansion method \cite{Eller2009, Valdivia2012}, the Fourier method \cite{Zhang2015, Wang2017, Wang2018} and the sampling method \cite{Alzaalig2019,  Griesmaier2017, Ji2019}.

In this article, we consider an inverse source problem for the time-harmonic elastic waves, where the elastic source term is modeled by the moment tensor point source. Our goal is to determine the moment tensor point source from the measured multi-frequency displacement data. The moment tensor is widely used in various practical scenarios and it covers a number of specific elastic sources such as earthquakes, explosions and mine collapses. Physically, the moment tensor represents the equivalent body force of an elastic source \cite{Aki2002}. It is noteworthy that the accurate estimation of the moment tensor point source plays a significantly crucial role in modern seismology. For example, the identification of moment tensor point sources is an essential ingredient in characterizing seismic events and the fault orientation\cite{Aki2002}.

Next, we present the mathematical formulation of the inverse source problem for the elastic waves. Let $\Omega\subset \mathbb{R}^d\,(d=2,3)$ be a bounded domain with $\mathcal{C}^2$  boundary $\Gamma=\partial \Omega$. Let $\omega$ be the angular frequency and $(\lambda, \mu)$ denote the Lam\'e constants such that $\mu>0$ and $d\lambda+2\mu>0$. Let the elastic wave be described by the displacement field $u\in \mathbb{C}^d(\mathbb{R}^d)$, then the propagation of elastic wave from the source term $ S $ is governed by the Lam\'e system:
\begin{equation}\label{eq:LameSystem}
   -\omega^2 u=\nabla\cdot\sigma(u)+S \quad \mathrm{in} \ \mathbb{R}^d,
 \end{equation}
where the stress tensor $\sigma$ is given by
\begin{equation*}
  \sigma(u)=\lambda \nabla\cdot u I+\mu \left(\,{\nabla u}^{\top}+\nabla u \,\right),
\end{equation*}
and $I$ is the $d\times d$ identity matrix. Here, the source term  $S$ consists of a finite number of well-separated moment tenser point sources, such that
\begin{equation}\label{eq:Source}
  S(x)= \sum_{j=1}^{m} M^{(j)} \nabla\delta(x-s_j),\quad m\in \mathbb{N}_+,\ s_j\in \Omega,
\end{equation}
where $M^{(j)} \in \mathbb{R}^{d \times d}$ signifies the two-rank moment tensor, and $\delta(x-s_j)$ denotes the Dirac delta distribution at point $s_j$,  $j=1,2\cdots,m$.

The inverse source problem that we are concerned with is to determine the locations $\{s_j\}_{j=1}^{m}$ and moment tensors $\{M^{(j)}\}_{j=1}^m$ from the displacement fields measured on the boundary $\Gamma$. It is remarked that the measurements $u|_{\Gamma}$ are linearly dependent on the source term $S$ but nonlinearly dependent on the locations and moment tensors of the point sources, and hence the inverse  problem is nonlinear and ill-posed.  Several inversion methods have been developed to reconstruct the moment tensor point source. Sj\"{o}green and Petersson \cite{Sjogreen2014} proposed a full waveform inversion method for estimating seismic source parameters by using time-dependent ground motion recordings at a number of receiver stations. In addition, Long, Motamed and Tempone \cite{Long2015} developed a fast method for seismic source inversion based on Bayesian statistics.  However,  only a single moment tensor point source is considered in both of the aforementioned works. For the inverse source problem of acoustic wave, a direct sampling method is recently proposed in \cite{Zhang2019} to recover multiple point sources. Motivated by the idea of \cite{Zhang2019}, in this work we propose a novel direct sampling scheme for reconstructing the elastic source $S$,  that is, identifying all the locations and moment tensors for the point sources.  To our best knowledge, this is the first sampling scheme to reconstruct multiple moment tensor point sources from  time-harmonic elastic measurements.  Interested readers could refer to \cite{Badia2015, Sara2018,  Badia2013, Ling2005} for some theoretical and computational results on recovering multiple point sources for the Poisson equation, the Helmholtz equation and the Maxwell system.

The proposed reconstruction schemes are based on some indicator functions which are associated with the asymptotic behaviors of the Bessel functions. The salient features of our approach are summarized as follows. First, the imaging schemes are very easy to implement with computational efficiency. Since only cheap integrations are involved in the formulation of the indicator functions, the schemes need neither iterative process nor forward solver. In fact, the indicator functions are formulated directly from the measured boundary data. Second, the sampling schemes utilize the data due to a discrete set of specially designated frequencies. These admissible frequencies are not uniformly distributed and have the capability of facilitating high accuracy of reconstruction. Third, the reconstruction procedure comprises two simple phases. In the first phase, the source locations are identified via the significant local maximizers of the indicator function. Then, in the second phase, the moment tensors can be recovered by the intensities of the indicator function evaluated at these source locations. Finally, the proposed method is robust with respect to the noise-contaminated measurement data.

The rest of the paper is organized as follows. Section \ref{sec:problem_setup} introduces the direct and inverse problems under consideration. In section \ref{sec:sampling method}, we present the direct sampling method with multi-frequency near-field measurements. Then we establish the uniqueness and stability in section \ref{sec:stability}. Section \ref{sec:example} is devoted to the numerical experiments, which verify the validity and robustness of the proposed reconstruction schemes.

\section{Problem setting}\label{sec:problem_setup}

In an isotropic homogeneous medium, the Lam$\acute{\mathrm{e}}$ system defined in \eqref{eq:LameSystem} reduces to the time-harmonic Navier equation (cf.\cite{Kupradze1979})
\begin{equation}\label{eq:NavierEquation}
   \omega^2 u+ \mathcal{L}_{\lambda,\mu} u  =-S, \quad \mathrm{in} \ \mathbb{R}^d,
\end{equation}
where $\mathcal{L}_{\lambda,\mu}$ denotes the linear isotropic elasticity operator
\begin{equation*}
  \mathcal{L}_{\lambda,\mu} u:=\mu \Delta u+(\lambda+\mu)\nabla\nabla\cdot u.
\end{equation*}
It is well known that any solution $u$ of \eqref{eq:NavierEquation} has the Helmholtz decomposition
\begin{equation*}
  u= u_p+ u_s,
\end{equation*}
where $u_p\in \mathbb{C}^d(\mathbb{R}^d)$ and $u_s\in \mathbb{C}^d(\mathbb{R}^d)$ denote the compressional (longitudinal) part and the shear (transversal) part of the displacement $u$, respectively. As usual, let $\mathrm{grad}^{\bot}:=(-\partial_2, \partial_1)^{\top}$ and  $\mathrm{div}^{\bot}:=(-\partial_2, \partial_1)$, then the compressional and shear part of the displacement field can be represented by
\begin{equation*}
u_p=-\frac{1}{k_p^2}\mathrm{grad\,div}\,  u, \quad
u_s=
\begin{cases}
-\dfrac{1}{k_s^2} \mathrm{grad}^{\bot}\,\mathrm{div}^{\bot}\,u, \quad \mathrm{in}\ \mathbb{R}^2, \\
\dfrac{1}{k_s^2} \mathrm{curl}\,\mathrm{curl}\, u, \quad
\quad\quad \mathrm{in}\ \mathbb{R}^3,
\end{cases}
\end{equation*}
where $k_p={\omega}/{\sqrt{\lambda+2\mu}}$ and $k_s={\omega}/{\sqrt{\mu}}$ are the compressional and shear wave numbers, respectively.

Given the source $S$ defined in \eqref{eq:Source}, the forward problem is to find the radiated field $u\in \left(H_{\rm loc}^1(\mathbb{R}^d)\right)^d$ such that
\begin{align}
 \label{eq:main} & \omega^2 u+ \mathcal{L}_{\lambda,\mu} u = -S, \quad \mathrm{in} \ \mathbb{R}^d, \\
 \label{eq:KSP} &  \lim_{r=|x|\rightarrow \infty} r^{\frac{d-1}{2}} \left(\dfrac{\partial u_p}{\partial r} -\mathrm{i}k_p u_p\right)= 0, \\
 \label{eq:KSS} & \lim_{r=|x| \rightarrow \infty} r^{\frac{d-1}{2}} \left(\dfrac{\partial u_s}{\partial r} -\mathrm{i}k_s  u_s\right)= 0,
\end{align}
where \eqref{eq:KSP} and \eqref{eq:KSS} are the Kupradze-Sommerfeld radiation conditions (cf.\cite{Kupradze1979}). The forward problem to \eqref{eq:main}-\eqref{eq:KSS} admits a unique solution, which can be represented by  (cf. \cite{Ammari2015})
\begin{equation*}
  u(x)=\int_{\Omega} G(x,y) S(y)\, \mathrm{d}y,
\end{equation*}
where $G$ is the Green tensor to the Navier equation
\begin{equation*}
  G(x,y)=\frac{k_s^2}{\omega^2} \Phi_{k_s}(x, y) I+
  \frac{1}{\omega^2}\nabla_x\nabla_x^{\top}
  \left(\Phi_{k_s}(x,y)-\Phi_{k_p}(x,y)\right).
\end{equation*}
Here $\Phi$ denotes the fundamental solution to the Helmholtz equation \cite{Colton2019}, that is
\begin{equation*}
  \Phi_{k}(x, y)=
  \begin{cases}
   \dfrac{\mathrm{i}}{4} H_0^{(1)}(k|x-y|), & d=2, \medskip \\
   \dfrac{\mathrm{e}^{\mathrm{i}k|x-y|}}{4\pi |x-y|}, & d=3,
  \end{cases}
\end{equation*}
where $H_0^{(1)}$ is the Hankel function of the first kind and order zero.

Let $\nu$ denote the exterior unit normal to the measurement surface $\Gamma$. The surface traction operator is then defined by
\begin{equation*}
  T_{\nu}:=
  \begin{cases}
  2\mu\nu\cdot \mathrm{grad} +\lambda\nu\mathrm{div}-\mu\nu^{\bot}\mathrm{div}^{\bot}, &  d=2,\medskip \\
  2\mu\nu\cdot \mathrm{grad} +\lambda\nu\mathrm{div}-\mu\nu\times \mathrm{curl}, & d=3.
  \end{cases}
\end{equation*}
To formulate the model inverse problem, we introduce the near-field dataset with a fixed frequency
\begin{equation}\label{eq:near-data}
  \left\{\left(u(\cdot, c_{\alpha}\omega), \,T_{\nu}u(\cdot, c_{\alpha}\omega)\right)\big|_{\Gamma}\right\},
\end{equation}
where $\alpha=p,s$ and
\begin{equation*}
c_{\alpha}:=
  \begin{cases}
      \sqrt{\lambda+2\mu}, & \alpha=p, \\
      \sqrt{\mu},                    & \alpha=s.
  \end{cases}
\end{equation*}

Now the inverse source problem under consideration can be stated as follows:
\begin{problem}
Given a finite number of admissible frequencies $\{\omega\}$,  determine the locations $\{s_j\}_{j=1}^m$ and the moment tensors  $\{M^{(j)}\}_{j=1}^m $ of the source $S$  from the measured multi-frequency  near-field data $\{\left(u(\cdot, c_{\alpha}\omega), \, T_{\nu}u(\cdot, c_{\alpha}\omega)\right)\big|_{\Gamma}\}$, where $\alpha=p,s$.
\end{problem}

We remark that the appropriate choice of the admissible frequencies would play a significantly important role in our inversion method. Details of this issue will be discussed in the next section.


\section{Direct sampling method}\label{sec:sampling method}

In this section, we propose a novel direct sampling method for determining the locations and moment tensors of the point sources from multi-frequency measurements. To this end, we first define the admissible set of frequencies. Under the mild a priori information that the lower bound of distances between distinct point sources is available, the admissible frequencies are defined as the following.
\begin{definition}[Admissible frequencies]\label{def:frequencies}
	Let $\omega^*$ be a fixed frequency such that
	\begin{equation}\label{eq:omega0}
		\omega^* \geq \frac{2}{ \displaystyle \min_{\substack{1\leq j, \, j' \leq m\\
					j\neq j'}} \mathrm{dist}(s_j, s_{j'})}.
	\end{equation}
Given $N\in \mathbb{N}_+$ and a fixed $\eta>1$, the admissible set of angular frequencies is given by
\begin{equation*}
	\mathbb{W}_N:= \left\{ \omega_n: \omega_n=\eta^{n-1}\omega^*, \  n=1,2,\cdots, N  \right\}.
\end{equation*}

\end{definition}

\begin{remark}
In terms of Definition \ref{def:frequencies}, it can be seen that these frequencies form a finite geometric series.  In particular, the selection of the lowest frequency $\omega^*$ as in \eqref{eq:omega0} is for the resolution purpose of separating different point sources so that the two closest sources are distinguishable. We would like to point out that these tailored and non-uniformly distributed frequencies have the advantage of fitting the justification of Theorem \ref{thm:main}.
\end{remark}

Based on the above admissible frequencies, we are going to introduce the multi-frequency indicator function. Denote by $\mathbb{S}^{d-1}$ the unit sphere in $\mathbb{R}^d$. For $\hat x \in \mathbb{S}^{d-1}, \omega\in \mathbb{R}_+$, let
\begin{equation*}
	\mathcal{R}(\hat x, \omega,\alpha):=\int_{\Gamma}
	\left\{[T_{\nu}V_{\alpha}(\hat x, y, c_{\alpha}\omega)]^{\top} u(y,c_{\alpha}\omega)
	-V_{\alpha}^{\top}(\hat x, y, c_{\alpha}\omega) T_{\nu} u(y,c_{\alpha}\omega)\right\}
	\mathrm{d}s(y).
\end{equation*}
where $\alpha=p, s$ and
\begin{equation*}
	V_{\alpha}(\hat x,y,\omega):=
	\begin{cases}
		\ \hat x\hat x^{\top}\mathrm{e}^{-\mathrm{i}\omega \hat x \cdot y/c_p }, & \alpha=p,\medskip\\
		\left(I-\hat x\hat x^{\top}\right)\mathrm{e}^{-\mathrm{i}\omega \hat x \cdot y/c_s }, & \alpha=s.
	\end{cases}
\end{equation*}
To characterize the unknown moment tensor sources at a general sampling point $z$, the following indicator function is proposed.
\begin{equation}\label{eq:NM}
    \mathcal{I}^{N, d}(z)=\frac{d}{ 2^{d-1}\pi \mathrm{i} N}\sum_{\omega\in \mathbb{W}_ N}\left(\frac{1}{\omega} \int_{\mathbb{S}^{d-1}}
     \left(\sum_{\alpha\in\{p,s\}} \mathcal{R}(\hat x, \omega,\alpha) \right)
   \otimes \hat x \mathrm{e}^{\mathrm{i}\omega\hat x\cdot z}   \,\mathrm{d}s(\hat x)\right),
\end{equation}
where $\otimes$ denotes the outer product and $\sum\limits_{\alpha\in\{p,s\}}$ signifies the sum over indices $p$ and $s$, i.e.,
\begin{equation*}
  \sum_{\alpha\in\{p,s\}} f(\cdot, \alpha):=f(\cdot, p)+f(\cdot, s).
\end{equation*}

We next analyze the multi-frequency indicator function defined in \eqref{eq:NM} and thus the following crucial lemma is needed.
\begin{lem}\label{lemma1}
Let $z\in \mathbb{R}^d\backslash\{0\}$, $\omega\in \mathbb{R}_{+}$ and $\hat z= z/|z|=(\hat z_1,\cdots, \hat z_d)$. Then it holds that
\begin{align*}
\int_{\mathbb{S}^{d-1}}  \hat x \otimes \hat x\, \mathrm{e}^{\mathrm{i}\omega\hat x\cdot z}\, \mathrm{d}s(\hat x)
=\begin{cases}
\pi\Big( J_0(\omega |z|)I- J_2(\omega|z|) A(\hat z) \Big), & d=2, \medskip \\
\dfrac{4\pi}{3}\Big(  j_0(\omega|z|)I+j_2(\omega|z|)\big(I-3\hat z \otimes \hat z\big)\Big), & d=3,
\end{cases}
\end{align*}
where $J_n$ and $j_n$ denote respectively the Bessel function and spherical Bessel function of order $n$, and
\begin{equation*}
    A(\hat z)=
    \begin{bmatrix}
        \hat z_2^2 -\hat z_1^2 & 2\hat z_1\hat z_2 \\
        2\hat z_1\hat z_2          & \hat z_1^2 -\hat z_2^2 \\
    \end{bmatrix}.
\end{equation*}
\end{lem}
\begin{proof}
For the two-dimensional case, let $\hat x=(\hat x_1, \hat x_2)=(\cos\theta, \sin\theta)^{\top}$ and $ z=|z|(\cos\phi, \sin\phi)^{\top}$. Using the Jacobi-Anger expansion in 2D,
\begin{equation*}
   \mathrm{e}^{\mathrm{i}\omega \hat x\cdot z}=J_0(\omega|z|)+2\sum_{n=1}^{\infty} \mathrm{i}^n J_n(\omega|z|)\cos(n\theta-n\phi),
\end{equation*}
and \cite[(2.9b), (2.9c) and (2.10)]{Zhang2019}, one can easily derive that
\begin{align*}
   \int_{\mathbb{S}^1}  \hat x \otimes \hat x\, \mathrm{e}^{\mathrm{i}\omega\hat x\cdot z}\, \mathrm{d}s(\hat x)
   & =\int_{\mathbb{S}^1}
   \begin{bmatrix}
     \hat x_1 \hat x_1 & \hat x_1 \hat x_2 \\
     \hat x_2 \hat x_1 & \hat x_2 \hat x_2
   \end{bmatrix}
   \mathrm{e}^{\mathrm{i}\omega\hat x\cdot z}\, \mathrm{d}s(\hat x)\\
   &=\pi\left(J_0(\omega |z|)I - J_2(\omega|z|)
   \begin{bmatrix}
     \cos 2\phi & \sin 2\phi \\
     \sin 2\phi & -\cos 2\phi
   \end{bmatrix} \right) \\
   & =\pi\left(J_0(\omega |z|) I- J_2(\omega|z|) A(\hat z) \right).
\end{align*}

Next, we discuss the three-dimensional case.  By \cite[Lemma 3.5]{Sara2018}, we find that
 \begin{equation*}
   \int_{\mathbb{S}^{2}}  \hat x \otimes \hat x\, P_n(\hat x\cdot \hat z)\, \mathrm{d}s(\hat x)=
    \begin{cases}
    \dfrac{4\pi}{3} I,\quad & n=0,\medskip \\
    \dfrac{4\pi}{5}\left( \hat z \otimes \hat z-\dfrac{1}{3} I\right),\quad & n=2,\medskip\\
   0, \quad & n\in \mathbb{N}\backslash\{0,2\},
   \end{cases}
 \end{equation*}
 where $P_n$ are the Legendre polynomials. Using the 3D Jacobi-Anger expansion
 \begin{equation*}
   \mathrm{e}^{\mathrm{i}\omega \hat x\cdot z}=
   \sum_{n=0}^{\infty} \mathrm{i}^n (2n+1) j_n(\omega|z|) P_n(\hat x \cdot \hat z ),
 \end{equation*}
 and a straightforward calculation, we have
 \begin{align*}
   \int_{\mathbb{S}^2}  \hat x \otimes \hat x\, \mathrm{e}^{\mathrm{i}\omega\hat x\cdot z}\, \mathrm{d}s(\hat x)
   & =\int_{\mathbb{S}^2}  \hat x \otimes \hat x \left(
   \sum_{n=0}^{\infty} \mathrm{i}^n (2n+1) j_n(\omega|z|) P_n(\hat x \cdot \hat z ) \right) \mathrm{d}s(\hat x) \\
   & =\frac{4\pi}{3}\Big(  j_0(\omega|z|)I+j_2(\omega|z|)\big(I-3\hat z \otimes \hat z\big)\Big).
 \end{align*}
\end{proof}

We are now ready to present the properties of the multi-frequency indicator function \eqref{eq:NM},  which plays an important role in reconstructing the locations and moment tensors of point sources. In the following, we denote by $B(x_0, R)$ the ball centered at $x_0$ with radius $R$, i.e., $B(x_0, R)=\{x\in \mathbb{R}^d:|x-x_0|<R\}$.
\begin{thm}\label{thm:main}
 Let the elastic source $S$ be of the form \eqref{eq:Source} with $M^{(j)} \neq 0$ and the
 indicator function be described in \eqref{eq:NM} with $\mathbb{W}_N$ defined in Definition \ref{def:frequencies}. We have the following asymptotic behavior of the indicator function
\begin{equation}\label{eq:00}
  \mathcal{I}^{N, d}(s_j)= M^{(j)} +\mathcal{O}\left(\frac{1}{N}\right),\quad N\to\infty,\quad j=1,2,\cdots m.
\end{equation}
Moreover, it holds that
\begin{align}
 \label{eq:01} & \left|\mathcal{I}_{\ell,\hbar}^{N, d}(z)\right| \leq |M_{\ell,\hbar}^{(j)} |+ \mathcal{O}\left(\frac{1}{N}\right),\quad \forall z\in B(s_j,1/\omega^*),\quad \ell,\hbar=1,\cdots,d,\\
  \label{eq:02} & \left|\mathcal{I}_{\ell,\hbar}^{N, d}(z)\right|=  \mathcal{O}\left(\frac{1}{N}\right), \quad \quad  \forall z\in \Omega\backslash \bigcup\limits_{j=1}^m B(s_j, 1/\omega^*),\quad \ell,\hbar=1,\cdots,d,
\end{align}
where $\mathcal{I}_{\ell,\hbar}^{N, d}$ and $M_{\ell,\hbar}^{(j)}$ refer to the element in row $\ell$ and column $\hbar$ of matrix $\mathcal{I}^{N, d}$ and $M^{(j)}$, respectively. In particular, the equality in \eqref{eq:01} holds if and only if $z=s_j$.
\end{thm}
\begin{proof}
Without loss of generality, we only consider the 3D case. In terms of
\begin{equation*}
   \int_{\mathbb{R}^3}\, g(x)\nabla\delta (x-a)\, \mathrm{d}x= - \nabla g(a),
\end{equation*}
and a straightforward calculation, we have
\begin{align*}
  \int_{\Omega} V_{\alpha}(\hat x,y, \omega)  S(y)\, \mathrm{d}y
  & = \int_{\Omega}\ V_{\alpha}(\hat x,y, \omega) \left(\sum_{j=1}^m M^{(j)} \nabla\delta(y-s_j)\right) \, \mathrm{d}y \\
 & = \sum_{j=1}^m \frac{\mathrm{i}\omega}{c_{\alpha}}\, V_{\alpha}(\hat x,s_j, \omega) (M^{(j)} \hat x),\quad \forall \omega\in\mathbb{W}_N,\quad \alpha=p,s.
\end{align*}

By Betti's integral theorem \cite{Kupradze1965}, multiplying $V_{\alpha}$ to equation  \eqref{eq:NavierEquation} and integrating over $\Omega$, we obtain
 \begin{equation*}
\int_{\Gamma} \Big(T_{\nu}V_{\alpha}(\hat x,y,\omega) u(y,\omega)- V_{\alpha}(\hat x,y, \omega)T_{\nu} u(y,\omega) \Big)\mathrm{d}s(y)
=\int_{\Omega} V_{\alpha}(\hat x, y,\omega) S(y) \, \mathrm{d}y,\ \forall \omega\in\mathbb{W}_N.
\end{equation*}
Combining the last two equations, it derives that
\begin{equation*}
  \sum_{\alpha\in\{p,s\}} \mathcal{R}(\hat x, \omega, \alpha)
=\sum_{j=1}^{m}\mathrm{i}\omega M^{(j)} \hat x \mathrm{e}^{-\mathrm{i}\omega\hat x\cdot s_j},\quad\forall \omega\in\mathbb{W}_N.
\end{equation*}
Therefore, it follows from Lemma \ref{lemma1} that
\begin{align*}
\mathcal{I}^{N,3}(z)
& = \frac{3}{ 4\pi \mathrm{i} N}\sum_{n=1}^N \left( \frac{1}{\omega_n} \int_{\mathbb{S}^{2}}
   \left(\sum_{\alpha\in\{p,s\}} \mathcal{R}(\hat x, \omega, \alpha) \right)
   \otimes \hat x \mathrm{e}^{\mathrm{i}\omega_n\hat x\cdot z}\,\mathrm{d}s(\hat x)\right),\\
& = \frac{3}{4\pi \mathrm{i}N}\sum_{n=1}^N  \int_{\mathbb{S}^{2}}
  \left( \sum_{j=1}^{m} \mathrm{i}\left( M^{(j)} \hat x \right)\otimes \hat x\mathrm{e}^{\mathrm{i}\omega_n\hat x\cdot (z-s_j)}\right) \,\mathrm{d}s(\hat x)\\
& = \frac{1}{N} \sum_{n=1}^N \sum_{j=1}^{m}\left(M^{(j)}  j_0(\omega_n|z-s_j|)
+M^{(j)} \left(I-3\,\hat{t}_j\otimes \hat{t}_j\right)j_2(\omega_n|z-s_j|)\right),
\end{align*}
where $\hat{t}_j=(z-s_j)/{|z-s_j|}$.

Furthermore, the last equation can be rewritten as
\begin{equation}\label{eq:multiJ}
\begin{aligned}
 \mathcal{I}^{N,3}(z)&=\sum_{j=1}^m  \left(M^{(j)}  \frac{1}{N}\sum_{n=1}^N j_0(\omega_n|z-s_j|)\right)\\
   &\quad + \sum_{j=1}^m \left(M^{(j)}  \left(I-3\,\hat t_j\otimes \hat t_j \right)   \frac{1}{N}\sum_{n=1}^N j_2(\omega_n|z-s_j|)\right).
   \end{aligned}
\end{equation}
Using the fact that
\begin{equation*}
  j_0(t)=\frac{\sin t}{t},\quad
  j_2(t)=\left(\frac{3}{t^2}-1\right)\frac{\sin t}{t}-\frac{3\cos t}{t^2}, \quad t>0,
\end{equation*}
and the Taylor expansion of $\sin t$ and $\cos t$, one has
\begin{align}
  \label{eq:expansionj0j2} & j_0(t)< 1-\frac{t^2}{3!}+\frac{t^4}{5!},\quad
  j_2(t)<\frac{23t^2-t^4}{5!}, \quad 0< t<1,\\
 \label{eq:formj0j2} & |j_0(t)|\leq \frac{1}{t},\quad
  |j_2(t)|<\frac{2}{t}, \quad t\geq 1.
  \end{align}
On the one hand, for $z\in \Omega\backslash \bigcup\limits_{j=1}^m B(s_j, 1/\omega^*)$, using \eqref{eq:formj0j2}, one can find that
\begin{equation}\label{eq:temp1}
\begin{aligned}
\frac{1}{N}\sum_{n=1}^N j_0(\eta^{n-1} \omega^* |z-s_{j}|)
     &\leq \frac{1}{N}\sum_{n=1}^{N}\frac{1}{\eta^{n-1}}
     =\mathcal{O}\left(\frac{1}{N}\right),\\
\frac{1}{N}\sum_{n=1}^{N} j_2(\eta^{n-1} \omega^* |z-s_{j}|)
   &\leq \frac{1}{N}\sum_{n=1}^{N}\frac{2}{\eta^{n-1}}
   =\mathcal{O}\left(\frac{1}{N}\right).
  \end{aligned}
\end{equation}
Substituting the $\eqref{eq:temp1}$ into \eqref{eq:multiJ}, we have
\begin{equation*}
  |\mathcal{I}_{\ell,\hbar}^{N,3}(z)|=\mathcal{O}\left(\frac{1}{N}\right), \quad \forall
  z\in \Omega\backslash \bigcup\limits_{j=1}^m B(s_j, 1/\omega^*), \quad \ell,\hbar=1,\cdots,d,
\end{equation*}
and this justifies \eqref{eq:02}.

On the other hand, for $z\in B(s_j, 1/\omega^*), \,j=1,2,\cdots, m$, in terms of \eqref{eq:omega0} and \eqref{eq:formj0j2},  equation \eqref{eq:multiJ} can be rewritten as
\begin{align}
 \mathcal{I}^{N,3}(z)
   = & M^{(j)}\frac{1}{N}\sum_{n=1}^N j_0(\omega_n|z-s_j|)\notag \\
   & +  M^{(j)}  \left(I-3\,\widehat{z-s_j}\otimes \widehat{z-s_j}\right) \frac{1}{N} \sum_{n=1}^N j_2(\omega_n|z-s_j|)+\mathcal{O}\left(\frac{1}{N}\right). \label{eq:multiJ1}
   \end{align}

Since $j_0(0)=1$ and $j_2(0)=0$, equation \eqref{eq:multiJ1} immediately yields
\begin{equation*}
  \mathcal{I}^{N,3}(s_j)=M^{(j)} +\mathcal{O}\left(\frac{1}{N}\right),
\end{equation*}
and thus \eqref{eq:00} holds. If $0<|z-s_j|\leq1/(\eta^{N-1}\omega^*)$, then by \eqref{eq:expansionj0j2}, we have
\begin{align}
  & \frac{1}{N}\sum_{n=1}^{N} j_0(\eta^{n-1} \omega^* |z-s_j|)
        <\!\frac{1}{N}\sum_{n=1}^{N}\left(1+\frac{\eta^{4(n-N)} }{5!}\right)=\!1+\mathcal{O}\left(\frac{1}{N}\right), \label{eq:j0_estimate1}\\
  & \frac{1}{N}\sum_{n=1}^{N} j_2(\eta^{n-1} \omega^* |z-s_j|)
        <\frac{1}{N}\sum_{n=1}^{N}\frac{23\eta^{2(n-N)}}{5!}=\mathcal{O}\left(\frac{1}{N}\right). \label{eq:j2_estimate1}
\end{align}
Finally, if $1/(\eta^{\widetilde{N}}\omega^*)< |z-s_j|\leq 1/(\eta^{\widetilde{N}-1}\omega^*),\, 1\leq\widetilde{N}\leq N-1$, then \eqref{eq:expansionj0j2} and \eqref{eq:formj0j2} imply that
\begin{align}
 \frac{1}{N}\sum_{n=1}^N j_0(\eta^{n-1} \omega^* |z-s_j|)
< & \frac{1}{N}\left(\sum_{n=1}^{\widetilde{N}}\left(1
+\frac{\eta^{4(n-\widetilde{N})}}{5!}\right)+\sum_{n=\widetilde{N}+1}^{N}
\frac{1}{\eta^{n-\widetilde{N}-1}}\right)\notag \\
= &\frac{\widetilde{N}}{N}+\mathcal{O}\left(\frac{1}{N}\right), \label{eq:j0_estimate2}
\end{align}
\begin{equation}
	\frac{1}{N}\sum_{n=1}^{N} j_2(\eta^{n-1} \omega^* |z-s_j|)
	< \frac{1}{N}\left(\sum_{n=1}^{\widetilde{N}}\frac{23\eta^{2(n-\widetilde{N})} }{5!}+\sum_{n=\widetilde{N}+1}^{N}\frac{2}{\eta^{n-\widetilde{N}-1}}\right)
	= \mathcal{O}\left(\frac{1}{N}\right). \label{eq:j2_estimate2}
\end{equation}
Hence, combining \eqref{eq:multiJ1} and \eqref{eq:j0_estimate1}-\eqref{eq:j2_estimate2}, we obtain
\begin{equation*}
   |\mathcal{I}_{\ell,\hbar}^{N,3} (z)| \leq |M_{\ell,\hbar}^{(j)} |+\mathcal{O}\left(\frac{1}{N}\right),\quad \forall z\in B(s_j, 1/\omega^*),
\end{equation*}
where the equality holds if and only if $z=s_j$.
\end{proof}

Roughly  speaking, for a sufficiently large $N$, Theorem \ref{thm:main} asserts that the indicator function almost achieves its maximum at the same location as the true source. This property is essential for the success of our reconstruction method.


\section{Uniqueness and stability}\label{sec:stability}

In this section, we give a uniqueness result for the inverse source problem and investigate the stability issue of the direct sampling method.
\begin{thm}
Let the source $S$ be of the form \eqref{eq:Source} and $\mathbb{W}_N$ be the admissible set of angular frequencies. Then $S$ can be uniquely determined by  the corresponding multi-frequency measurements $\left\{\left(u(\cdot, c_{\alpha}\omega), \,T_{\nu}u(\cdot, c_{\alpha}\omega)\right)\big|_{\Gamma}\right\}$, $\omega\in \mathbb{W}_N$, $\alpha=p,s$, when $N \rightarrow +\infty$.
\end{thm}
\begin{proof}
Without loss of generality, we only need to consider the homogeneous boundary value problem. Suppose that
\begin{equation*}
  u(\cdot,c_p\omega)|_{\Gamma}= T_{\nu}u(\cdot,c_p \omega)|_{\Gamma}= u(\cdot, c_s\omega)|_{\Gamma}= T_{\nu}u(\cdot, c_s\omega)|_{\Gamma}=0,\quad \forall \omega\in\mathbb{W}_N,
\end{equation*}
then we have
\begin{equation*}
  \mathcal{R}(\hat x, \omega, p)=\mathcal{R}(\hat x, \omega, s)=0, \quad \forall \omega\in\mathbb{W}_N,
\end{equation*}
which, together with \eqref{eq:NM}, yields
\begin{equation*}
  \mathcal{I}^{N,d}(z)=0, \quad z\in \Omega.
\end{equation*}
Combining the last equation and \eqref{eq:00}, we can get
\begin{equation*}
   M^{(j)} =\mathcal{O}\left(\frac{1}{N}\right), \quad j=1,2,\cdots, m.
\end{equation*}
Furthermore, let $N \to +\infty$, we have $M^{(j)} =0$ for $j=1,2,\cdots, m$, and it completes the proof of this theorem.
\end{proof}

Now let us discuss the stability.  Assume that the measured noisy  data  satisfies
\begin{equation}\label{eq:noise}
  \begin{aligned}
 \|u^{\epsilon}(\cdot,c_\alpha\omega)-u(\cdot,c_\alpha\omega)\|_{L^2(\Gamma)}\leq & \epsilon\|u(\cdot,c_\alpha\omega)\|_{L^2(\Gamma)}, \quad \forall \omega\in\mathbb{W}_N,\\
 \|T_{\nu}u^{\epsilon}(\cdot,c_\alpha\omega)-T_{\nu}u(\cdot,c_\alpha\omega)\|_{L^2(\Gamma)}\leq  & \epsilon\|T_{\nu}u(\cdot,c_\alpha\omega)\|_{L^2(\Gamma)}, \quad \forall \omega\in\mathbb{W}_N,\\
  \end{aligned}
\end{equation}
where $\alpha=p,s$, and $\epsilon>0$ denotes the noise level. Introduce the perturbed indicator function by
\begin{equation}\label{eq:NM-noise}
	\mathcal{I}^{N, d, \epsilon}(z)=\frac{d}{ 2^{d-1}\pi \mathrm{i} N}\sum_{\omega\in \mathbb{W}_ N}\left(\frac{1}{\omega} \int_{\mathbb{S}^{d-1}}
\left(\sum_{\alpha\in\{p,s\}} \mathcal{R}^\epsilon(\hat x, \omega,\alpha) \right)
\otimes \hat x \mathrm{e}^{\mathrm{i}\omega\hat x\cdot z}\,\mathrm{d}s(\hat x)\right).
\end{equation}
Here, the function $\mathcal{R}^{\epsilon}$ is given by
\begin{equation}\label{eq:R-epsilon}
 \mathcal{R}^{\epsilon}(\hat x, \omega, \alpha)=\int_{\Gamma}
 [T_{\nu}V_{\alpha}(\hat x, y, c_{\alpha}\omega)]^{\top} u^{\epsilon}(y,c_{\alpha}\omega)
 -V_{\alpha}^{\top}(\hat x, y, c_{\alpha}\omega) T_{\nu} u^{\epsilon}(y,c_{\alpha}\omega)\,
 \mathrm{d}s(y).
\end{equation}

\begin{thm}
Let the elastic source $S$ be of the form \eqref{eq:Source} with $M^{(j)} \neq 0$ and the indicator functions  $\mathcal{I}^{N, d, \epsilon}$ be describes in \eqref{eq:NM-noise}. For a sufficiently large $N$, we have  the following asymptotic expansion
\begin{equation*}
  \mathcal{I}^{N, d, \epsilon}(s_j)= M^{(j)} +\mathcal{O}\left(\frac{1}{N}\right)+\mathcal{O}(\epsilon),\quad j=1,2,\cdots m.
\end{equation*}
Furthermore, we have
\begin{align}
&\left|\mathcal{I}_{\ell,\hbar}^{N, d, \epsilon}(z)\right|\leq|M_{\ell,\hbar}^{(j)} | + \mathcal{O}\left(\frac{1}{N}\right)+\mathcal{O}(\epsilon), \ \forall z\in B(s_j,1/\omega^*),\quad \ell,\hbar=1,\cdots,d, \label{eq:01noise}\\
&\left|\mathcal{I}_{\ell,\hbar}^{N, d, \epsilon}(z)\right|=  \mathcal{O}\left(\frac{1}{N}\right)+\mathcal{O}(\epsilon), \quad \forall z\in \Omega\backslash \bigcup\limits_{j=1}^{m} B(s_j, 1/\omega^*), \quad \ell,\hbar=1,\cdots,d,\notag
\end{align}
where the equality \eqref{eq:01noise} holds only at $z=s_j$.

\end{thm}
\begin{proof}
From \eqref{eq:noise} and the definition of $\mathcal{R}^{\epsilon}$ in \eqref{eq:R-epsilon}, it can be readily seen that
\begin{equation*}
  |\mathcal{R}^{\epsilon}(\hat x, \omega, \alpha)-\mathcal{R}(\hat x, \omega,\alpha)|=
  \mathcal{O}(\epsilon), \quad \alpha=p,s.
\end{equation*}
Thus, the proof is completed by a similar argument as the proof of Theorem \ref{thm:main}.
\end{proof}

Finally, we present the reconstruction scheme in the following {\bf Algorithm}.
\begin{table}[htp]
\centering
\begin{tabular}{cp{.8\textwidth}}
\toprule
\multicolumn{2}{l}{{\bf Algorithm:}\quad Reconstruction of elastic sources with multi-frequency data} \\
\midrule
 {\bf Step 1} & Given the number of the frequencies $N$ and a fixed $\eta>1$, choose  the angular frequency set $\mathbb{W}_N$ in Definition \ref{def:frequencies} and collect the noisy near-field data in \eqref{eq:near-data}. \\
{\bf Step 2} & Select a sampling mesh  $\mathcal{T}_h$  in $\Omega$. For each sampling point $z\in \mathcal{T}_h$, evaluate the perturbed indicator function $\mathcal{I}^{N,d,\epsilon}(z)$ in \eqref{eq:NM-noise}.\\
 {\bf Step 3} & According to the values of $\sum_{\ell,\hbar=1}^d|\mathcal{I}_{\ell,\hbar}^{N,d,\epsilon}(z)|^2$, collect the significant local maximizers $\{\widetilde{z}_{j}\}_{j=1}^m$, then $\{\widetilde{z}_{j}\}_{j=1}^m$ is treated as the locations of the point sources. \\
{\bf Step 4} & By substituting $\{\widetilde{z}_{j}\}_{j=1}^m$ into the indicator function $\mathcal{I}^{N,d,\epsilon}(\widetilde{z}_j)$ in \eqref{eq:NM-noise},  we obtain the reconstructed moment tensors as
$M^{(j)} \approx \mathcal{I}^{N,d,\epsilon}(\widetilde{z}_j)$. \\
\bottomrule
\end{tabular}
\end{table}


\section{Numerical experiments}\label{sec:example}

In this section, we present several two and three dimensional numerical examples to illustrate the effectiveness and robustness of the proposed method.

In all the numerical examples, we consider the domain $\Omega=[-6, 6]^d$ and Lam\'{e} constants $\lambda=1$ and $\mu=1$. The synthetic Cauchy datasets $(u,T_{\nu} u)$ are generated by solving the forward problem of \eqref{eq:main}-\eqref{eq:KSS} via direct integration.  For the two-dimensional case, the measurement curve $\Gamma$ is chosen as a circle centered at the origin with radius $R=10$, and $200$ measurement points are uniformly distributed on $\Gamma$. For the three-dimensional  case, we choose 500 pseudo-uniformly distributed measurement directions on the sphere $\Gamma$ with radius $R=10$. Figure \ref{fig:Geometry_2D} shows the two-dimensional geometrical setting of the problem, where the locations of  true sources are marked by the small red points,  the measurement curve is plotted as the blue circle and the sampling domain is marked by the dotted square.

\begin{figure}
    \centering \includegraphics[width=0.4\textwidth]{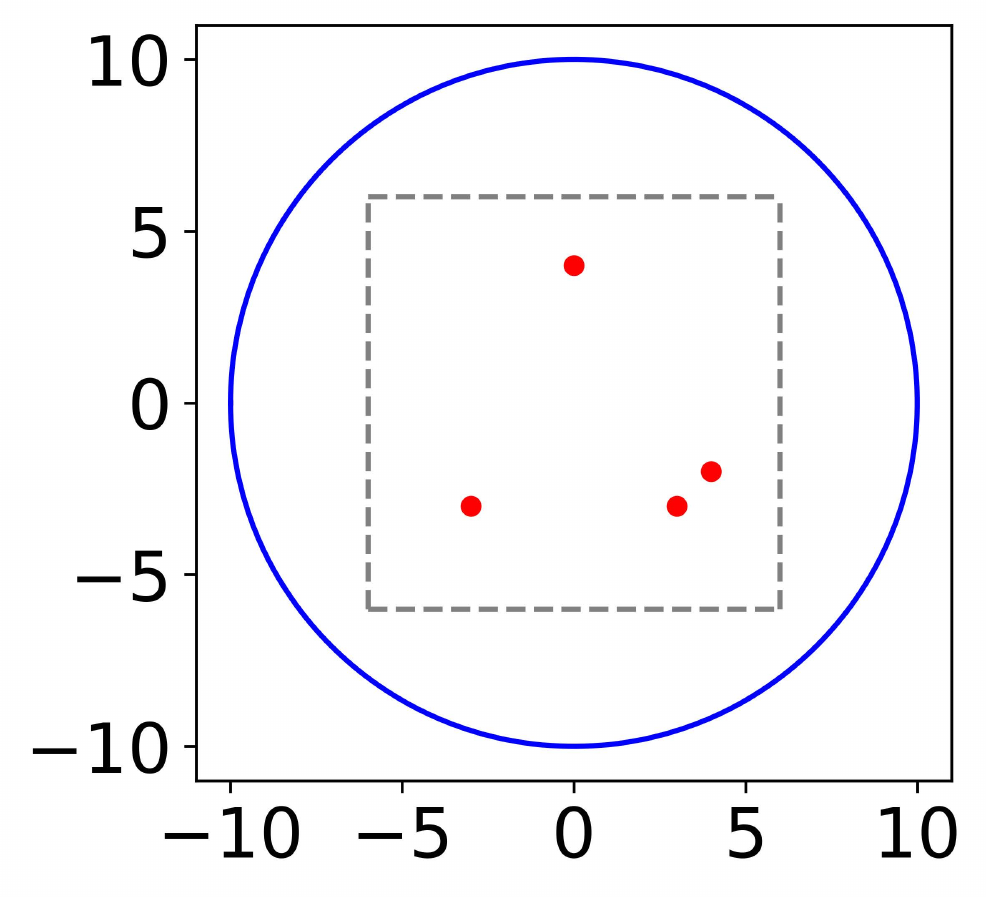}\\
    \caption{Geometrical setting of the two-dimensional problem.}\label{fig:Geometry_2D}
\end{figure}

Next, we present the details of synthetic data. Let $\eta=1.2$ and $\omega^*=5$, unless otherwise specified, the number of frequencies is taken as $N=10$, then the angular frequency sets are given by
\begin{equation*}
\mathbb{W}_{10}= \left\{\omega_n:\ \omega_n=5\times1.2^{n-1},\ n= 1, 2, \cdots 10 \right\}.
\end{equation*}
With the aforementioned admissible angular frequencies, the near field data can be written as
\begin{equation*}
  \left\{(\ u(x_i,\omega_n\sqrt{\lambda+2\mu})
   ,\ T_{\nu}u(x_i,\omega_n\sqrt{\lambda+2\mu} ), u(x_i,\omega_n\sqrt{\mu})
   ,\ T_{\nu}u(x_i,\omega_n\sqrt{\mu})\ )\right\}, \\
\end{equation*}
where $x_i\in\Gamma, \omega_n\in \mathbb{W}_{10}, i=1,2,\cdots,200$, and $n=1,2,\cdots, 10$.

To test the stability of the proposed approach, some random perturbations are added to the synthetic data. Let $u=(u_1, \cdots, u_d)^{\top}$ and $T_{\nu}u=(T_{\nu}u_1, \cdots, T_{\nu}u_d)^{\top}$, then the noisy data were given by
\begin{align*}
 u_{\tau}^{\epsilon}= & u_{\tau}+\epsilon r_{1,{\tau}}|u_{\tau}|\mathrm{e}^{\mathrm{i}\pi r_{2,{\tau}}},\\
 T_{\nu}u_{\tau}^{\epsilon}= & T_{\nu}u_{\tau}+\epsilon r_{1,{\tau}}|T_{\nu}u_{\tau}|\mathrm{e}^{\mathrm{i}\pi r_{2,{\tau}}},
 \end{align*}
where $\tau=1,\cdots, d$, $r_{1,\tau}$ and $r_{2,\tau}$ are two uniformly distributed random numbers, both ranging from $-1$ to 1, and $\epsilon>0$ represents the noise level.


\begin{figure}
    \subfigure[]{\includegraphics[width=0.3\textwidth]{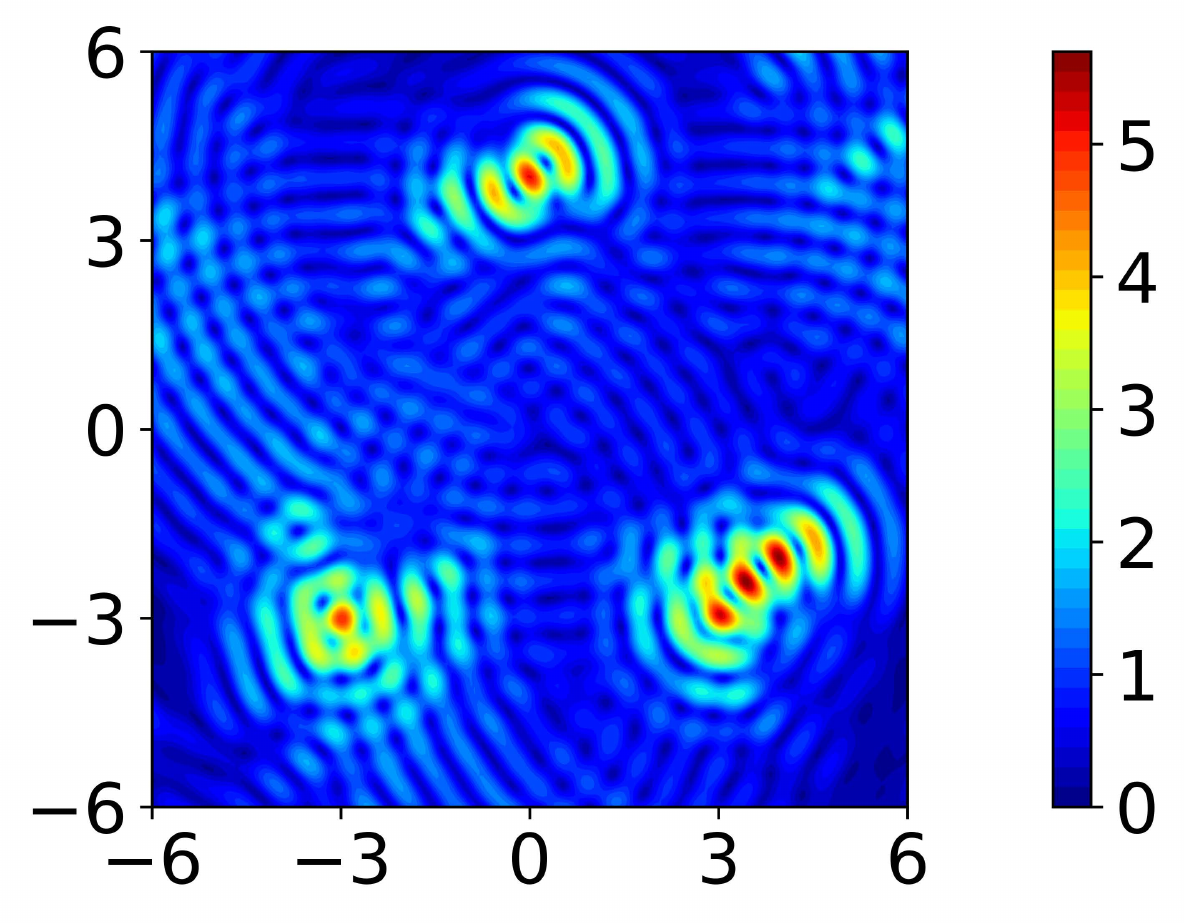}}
    \subfigure[]{\includegraphics[width=0.3\textwidth]{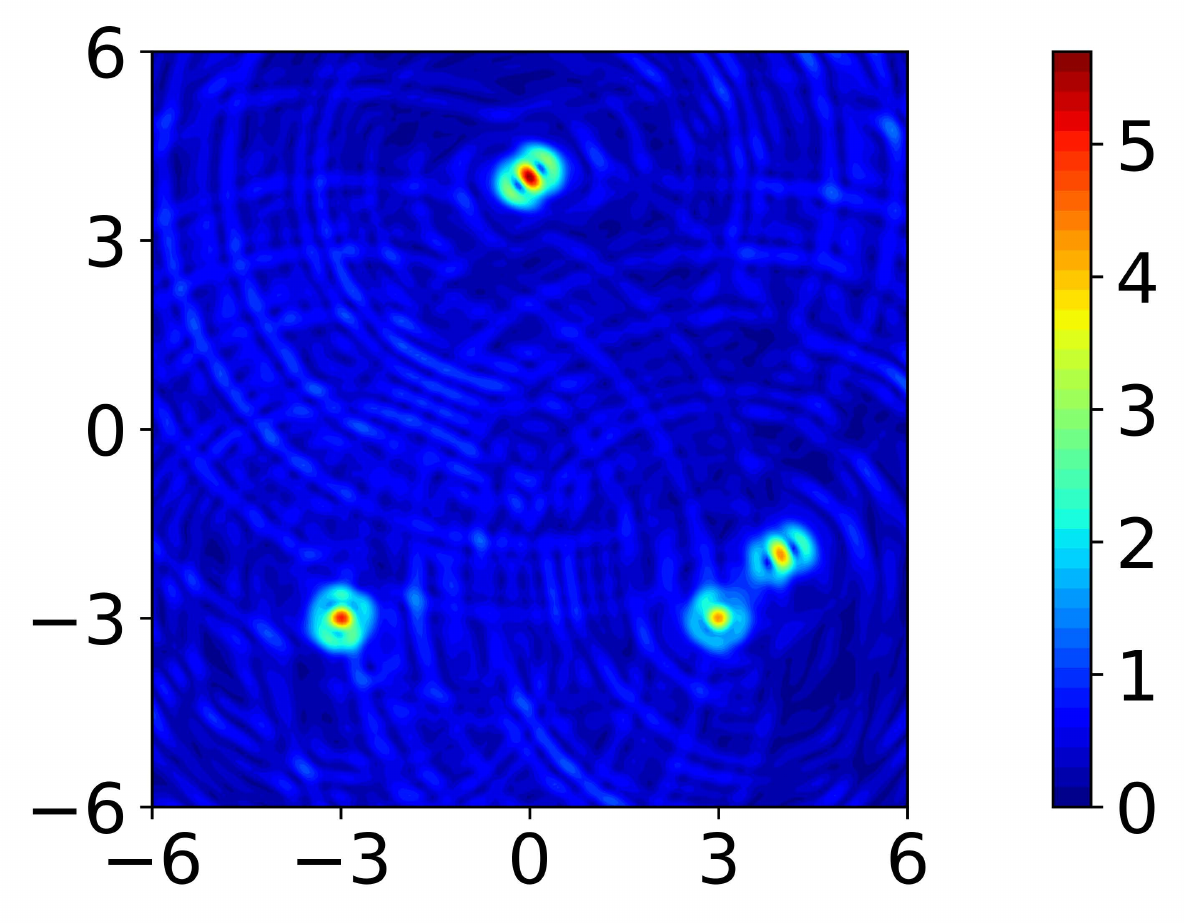}}
    \subfigure[]{\includegraphics[width=0.3\textwidth]{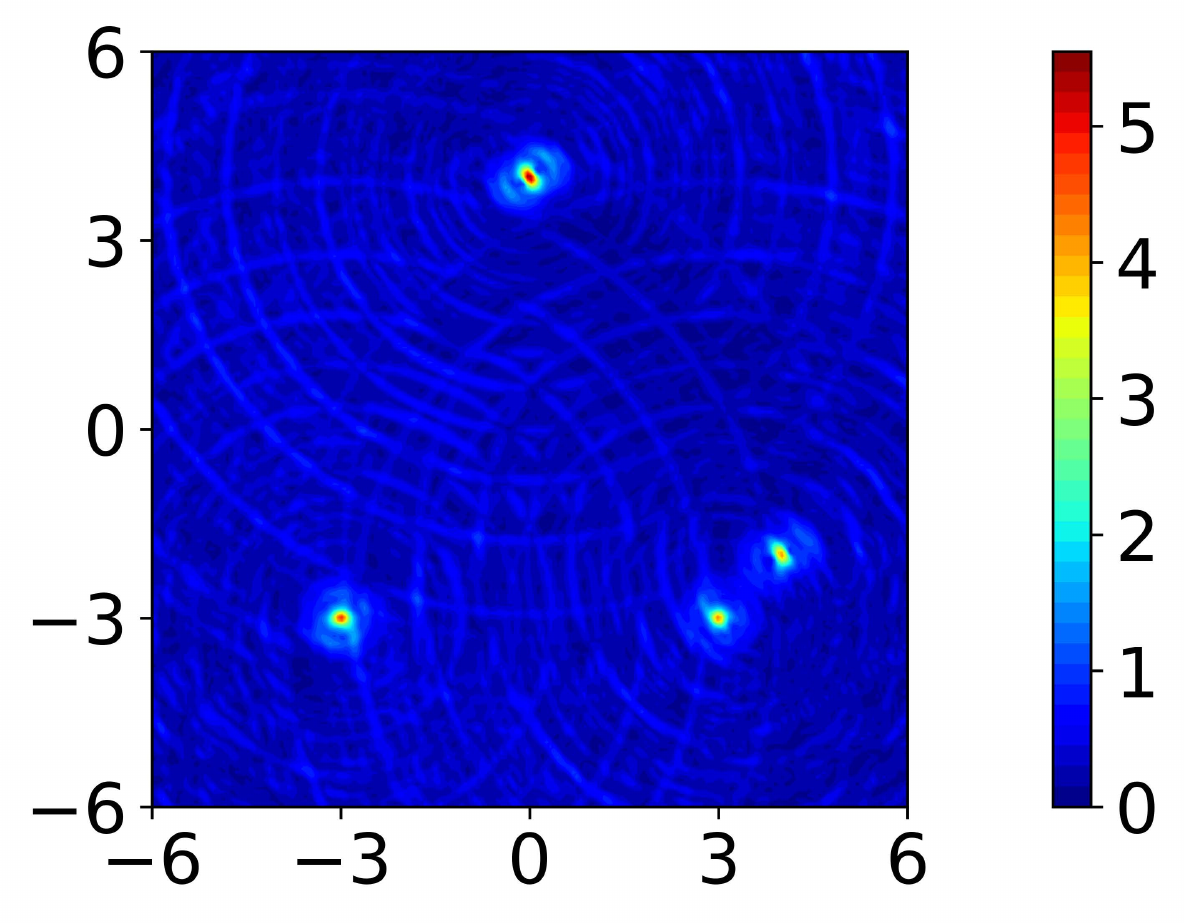}}
    \caption{\label{fig:Location_2D_multi} Contour plots of the multi-frequency indicator function $\sum_{\ell,\hbar=1}^2|\mathcal{I}_{\ell,\hbar}^{N,2,\epsilon}(z)|^2$  with different numbers of frequencies $N$. (a) $N=2$, (b) $N=6$, (c) $N=10$.}
\end{figure}

\begin{figure}
    \subfigure[]{\includegraphics[width=0.3\textwidth]{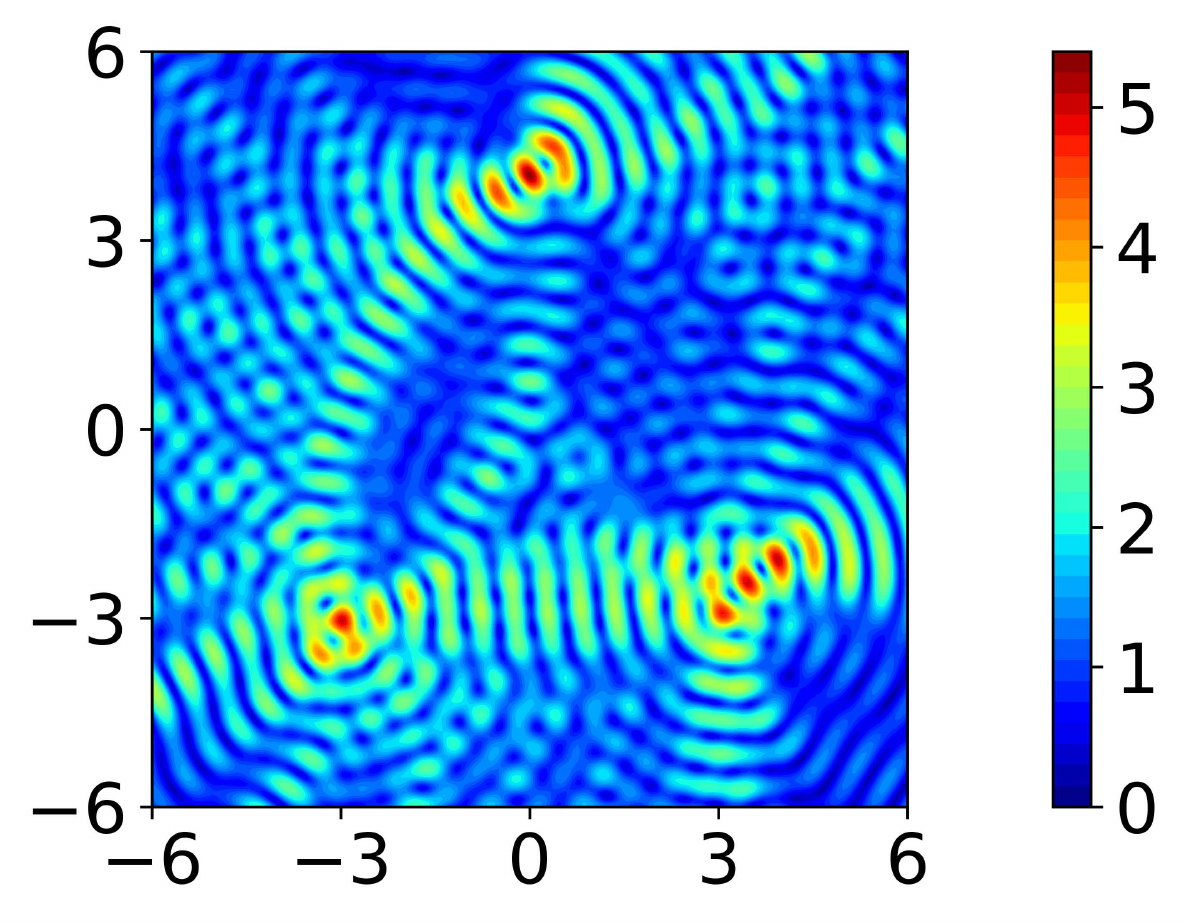}}
    \subfigure[]{\includegraphics[width=0.3\textwidth]{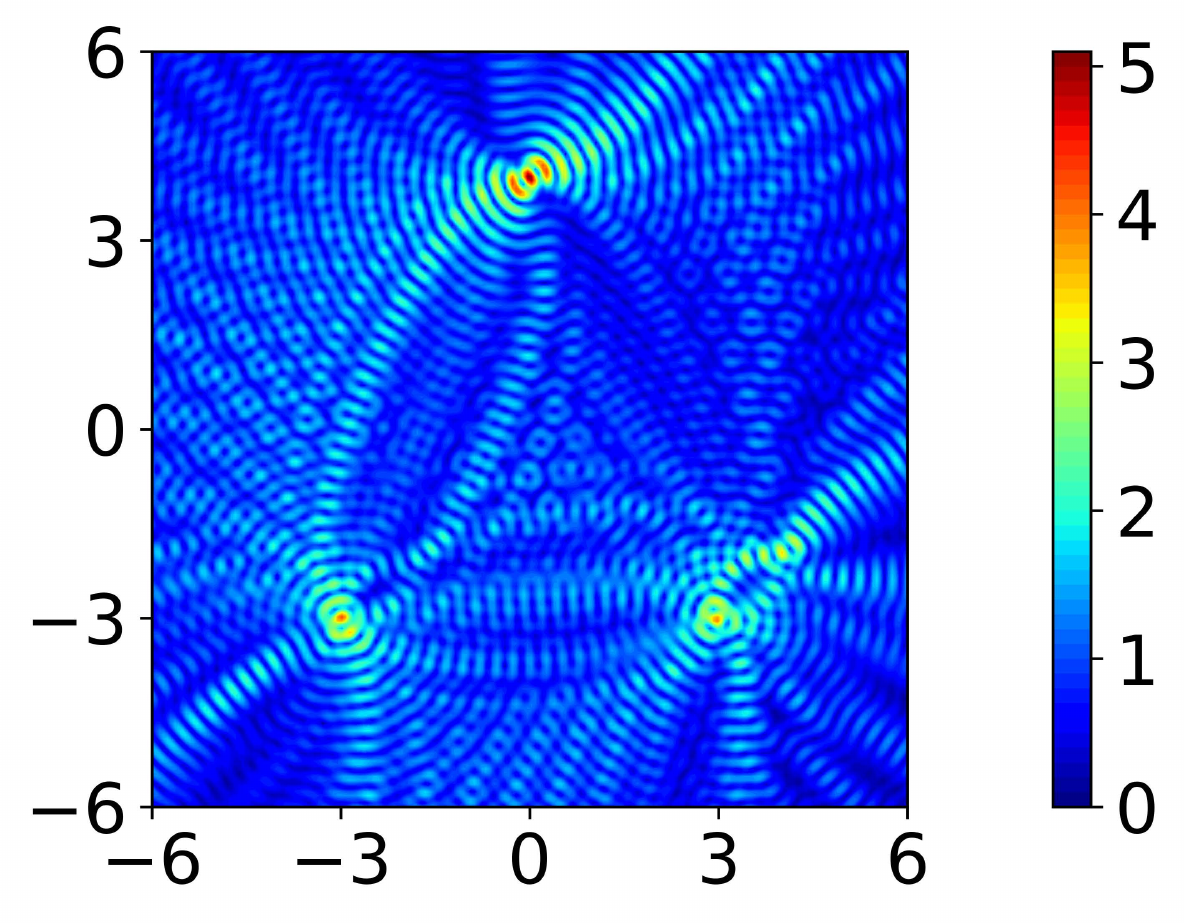}}
    \subfigure[]{\includegraphics[width=0.3\textwidth]{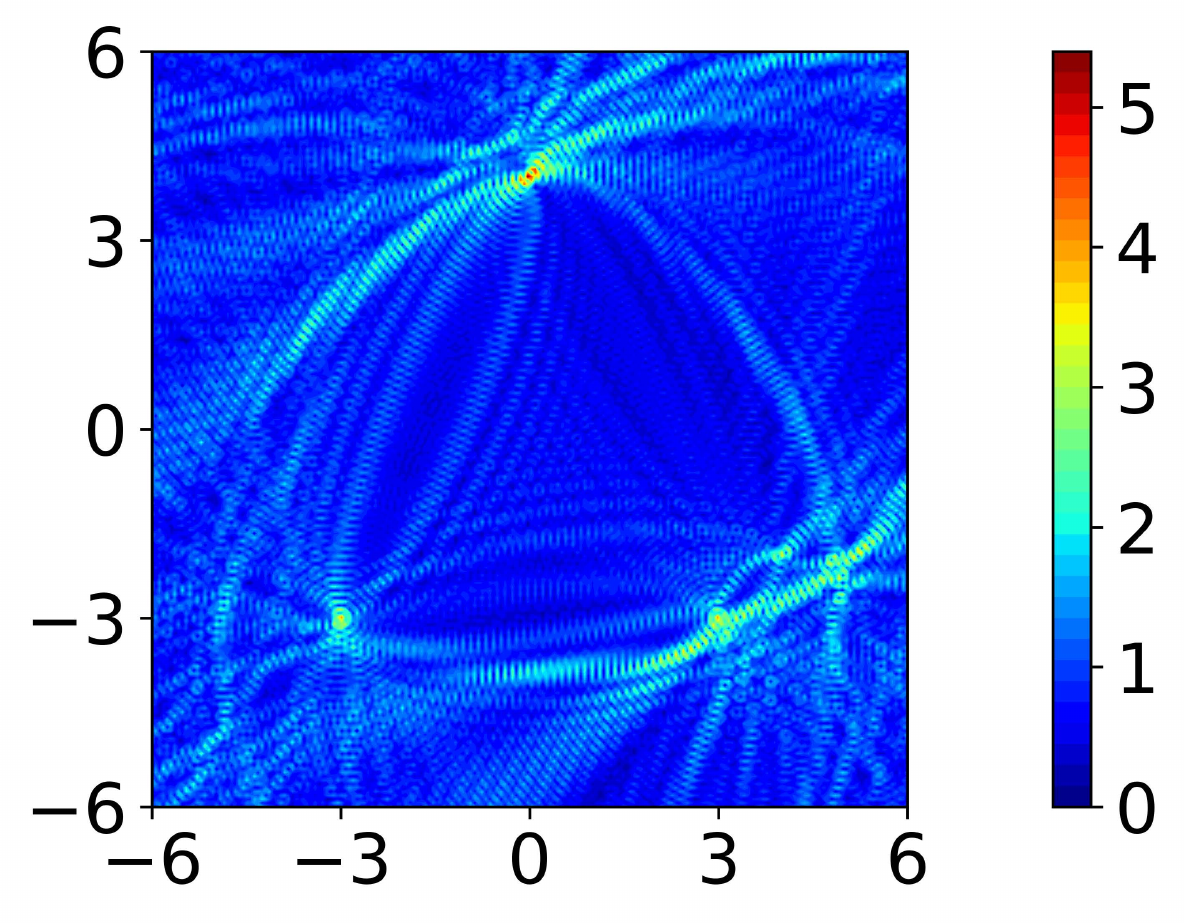}}
    \caption{\label{fig:Location_2D_single} Contour plots of the single-frequency indicator function $\sum_{\ell,\hbar=1}^2|\widehat{\mathcal{I}}_{\ell,\hbar}^{n,2,\epsilon}(z)|^2$  with different frequencies $\omega_n$. (a) $n=2$, (b) $n=6$, (c) $n=10$.}
\end{figure}

\begin{table}
	\centering
	\begin{tabular}{lllll}
		\toprule
		& \multicolumn{2}{c}{\underline{\qquad Exact sources \qquad}}
		& \multicolumn{2}{c}{\underline{\qquad Reconstructed sources \qquad}} \\
	    & \multicolumn{1}{c}{Location} & \multicolumn{1}{c}{Moment tensor}
	    & \multicolumn{1}{c}{Location} & \multicolumn{1}{c}{Moment tensor}\\
		\midrule
    	& $(0, 4)$
        & $ \begin{bmatrix}
                 4 & 2 \\
                 2 & 3 \\
            \end{bmatrix}$
        & $(-0.03,4.01)$
        & $ \begin{bmatrix}
               3.72 & 1.92 \\
                2.05 & 2.82 \\
            \end{bmatrix}$ \medskip \\
	    & $(-3,-3)$
        & $ \begin{bmatrix}
               -3  & 0 \\
                0 & -4 \\
            \end{bmatrix}$
        & $ (-2.98, -2.98)$
        & $ \begin{bmatrix}
                -2.96 & -0.03 \\
               0.09 & -3.82 \\
            \end{bmatrix}$ \medskip \\
		& $(3,-3)$
        & $ \begin{bmatrix}
                0& 3 \\
                 3 & 0 \\
            \end{bmatrix}$
        & $(2.98,-2.98)$
        & $ \begin{bmatrix}
                 0.23  & 2.97 \\
                3.06 & 0.145 \\
            \end{bmatrix}$ \medskip \\
		& $(4,-2)$
        & $ \begin{bmatrix}
                3 & 2 \\
                  2 & 0 \\
            \end{bmatrix}$
        & $( 4.01,  -2.02)$
        & $ \begin{bmatrix}
                3.04  & 2.01  \\
                2.02  & 0.07 \\
            \end{bmatrix}$ \medskip \\
		\bottomrule
	\end{tabular}
	\caption{Reconstruction of four moment tensor point sources from multi-frequency near-field data in 2D.}\label{tab:2D}
\end{table}

\begin{example}
In the first example, we will reconstruct the locations and strengths of four moment tensor point sources with multi-frequency near-field data in the two dimensions. Some parameters of the locations and the strengths  are presented in Table \ref{tab:2D}. Here, we take a uniformly distributed $200\times 200$ sampling mesh $\mathcal{T}_h$ over the  global sampling domain $\Omega=[-6,6]\times[-6,6]$.
To verify the stability of the proposed schemes, $5\%$ noise was added to the artificial multi-frequency near-field data.
\end{example}

Figure \ref{fig:Location_2D_multi} presents  the indicator function $\sum_{\ell,\hbar=1}^2|\mathcal{I}_{\ell,\hbar}^{N,2,\epsilon}(z)|^2$  with different numbers of the frequencies. It can be seen that the indicator function attains a significant local maximum near the exact locations of the point sources (see Figure \ref{fig:Geometry_2D}). By comparing the imaging results among Figure \ref{fig:Location_2D_multi}\,(a)-(c), one can observe that the reconstructed locations are more close to the exact locations as $N$ increases. To exhibit the accuracy quantitatively, we list the parameters (locations and moment tensors) of  the exact and reconstructed sources, respectively, in Table \ref{tab:2D}.

If only a single frequency is utilized, i.e., given a fixed $n$, then the perturbed multi-frequency indicator function \eqref{eq:NM-noise} reduces to the following single-frequency version
\begin{equation*}
   \widehat{ \mathcal{I}}^{n,d,\epsilon}(z)=\frac{d}{ 2^{d-1} \pi\mathrm{i}} \cdot \frac{1}{\omega_n} \int_{\mathbb{S}^{d-1}}
     \left(\sum_{\alpha\in\{p,s\}} \mathcal{R}^{\epsilon}(\hat x, \omega_n,\alpha) \right)
   \otimes \hat x \mathrm{e}^{\mathrm{i}\omega_n\hat x\cdot z}   \,\mathrm{d}s(\hat x), \ n=1,2, \cdots.
\end{equation*}
Figure \ref{fig:Location_2D_single} shows the single-frequency  indicator function $\sum_{\ell,\hbar=1}^2|\widehat{\mathcal{I}}_{\ell,\hbar}^{n,2,\epsilon}(z)|^2$ with different frequencies $\omega_n$.   Comparing  Figure  \ref{fig:Location_2D_multi} and Figure  \ref{fig:Location_2D_single}, it  is shown that the multi-frequency version could yield better reconstructions for determining  locations of the point sources.


\begin{example}
In this example, we aim to recover the locations and strengths of three moment tensor point sources with multi-frequency data in the three dimensions. The relevant parameters of the locations and strengths  are presented in Table \ref{tab:3D}. To illustrate the stability, $5\%$ noise was also added to the synthetic data.
\end{example}

Following the idea of the two-level sampling strategy in \cite{Zhang2019}, we also adopt a coarse-to-fine scheme  in order to decrease the overall computational cost.  To this end, we first use a relatively coarse global sampling grid of $50\times 50\times 50$ to roughly identify the locations. Then we take a uniform local fine sampling grid of $50\times 50\times 50$ with  side-length $2/\omega^*$ for fine tuning. The reconstructions of the locations are shown in Figure \ref{fig:Location_3D}, it is clear that the reconstruction improves as $N$ increases.
To demonstrate the accuracy of the reconstructions, the reconstructed locations and moment tensors are listed in Table \ref{tab:3D}.

\begin{table}[h]
	\centering
\renewcommand\tabcolsep{1.0pt}
    \begin{threeparttable}
	\begin{tabular}{lllll}
		\toprule
		& \multicolumn{2}{c}{\underline{\qquad Exact sources \qquad}}
		& \multicolumn{2}{c}{\underline{\qquad Reconstructed sources \qquad}} \\
    	& \multicolumn{1}{c}{Location} & \multicolumn{1}{c}{Moment tensor}
	    & \multicolumn{1}{c}{Location} & \multicolumn{1}{c}{Moment tensor}\\
		\midrule
		& $(4, 4, 4)$
        & $\begin{bmatrix}
                 9 & -11 & 10 \\
               -10 & -11 &  9 \\
               -10 & 11  & -9 \\
           \end{bmatrix}$
        & $(3.99,3.99,3.99)$
        & $ \begin{bmatrix}
               8.91 & -11.03 & 9.83 \\
              -9.77 & -10.92  & 8.85 \\
             -10.21 & 11.44  & -8.98 \\
            \end{bmatrix}$ \medskip \\
		& $(-4,4,-2)$
        & $ \begin{bmatrix}
                9  & 10 &  11 \\
                10 & -9 & -11 \\
                -9 & 10 &  11 \\
            \end{bmatrix}$
        & $ (-3.99,3.99,-2.00)$
        & $ \begin{bmatrix}
                8.72  & 9.62  & 10.77 \\
                10.21 & -8.96 & -10.74 \\
                -8.70 & 10.34 & 11.17 \\
            \end{bmatrix}$ \medskip \\
		& $(2,-3,-4)$
        & $ \begin{bmatrix}
                -11 & 10 & -9 \\
                9   & 10 & 10 \\
                10  & -11 & 9 \\
            \end{bmatrix}$
        & $(2.00,-3.00,-3.99)$
        & $ \begin{bmatrix}
              -11.11 & 10.39  & -8.81 \\
                8.69 & 9.62   & 9.69 \\
               10.41 & -11.17 & 9.20 \\
            \end{bmatrix}$ \medskip \\
		\bottomrule
	\end{tabular}
\end{threeparttable}
	\caption{Reconstruction of three moment tensor point sources from multi-frequency far-field data in 3D.}\label{tab:3D}
\end{table}

\begin{figure}
    \subfigure[]{\includegraphics[width=0.48\textwidth]{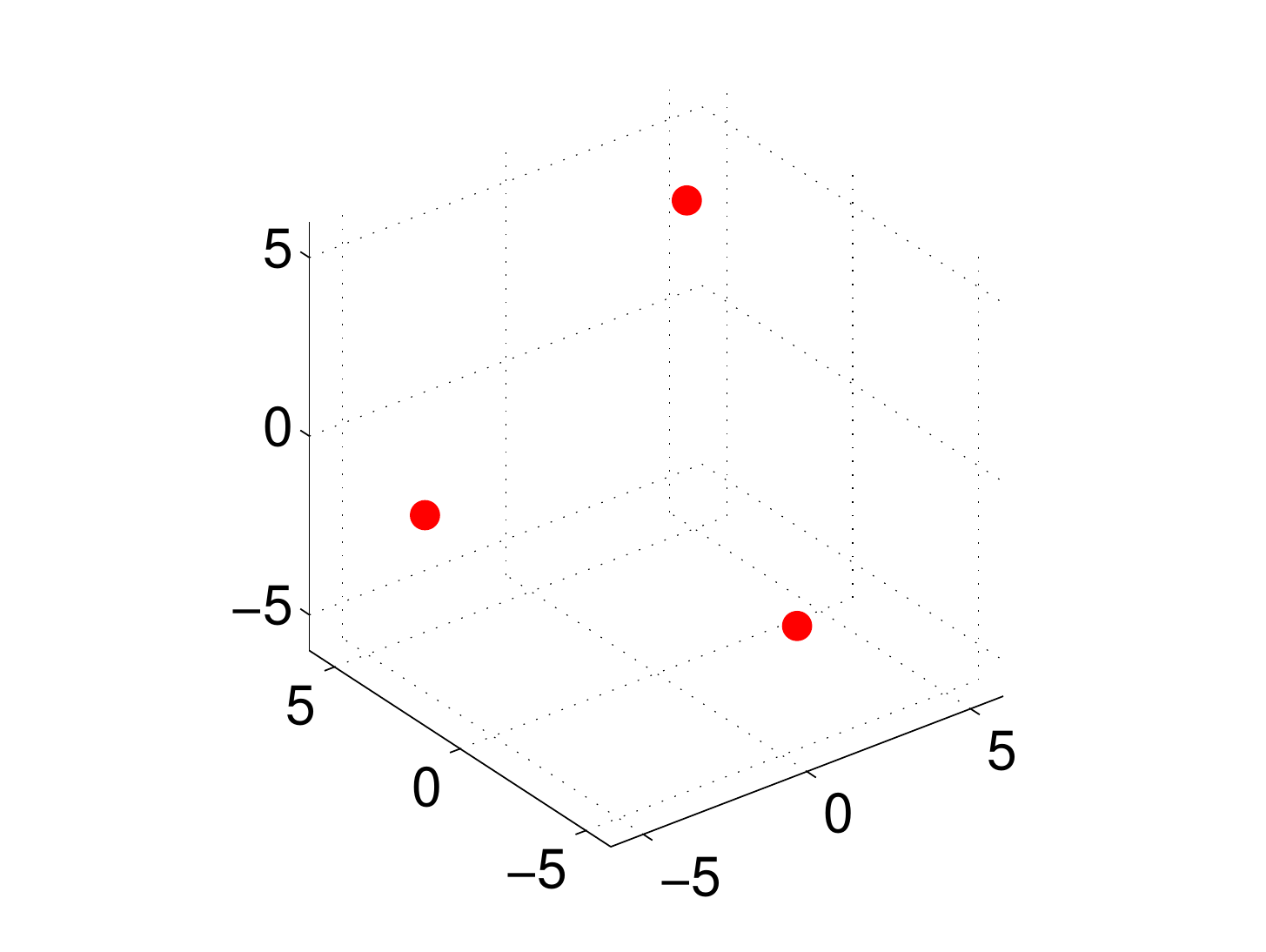}}
    \subfigure[]{\includegraphics[width=0.48\textwidth]{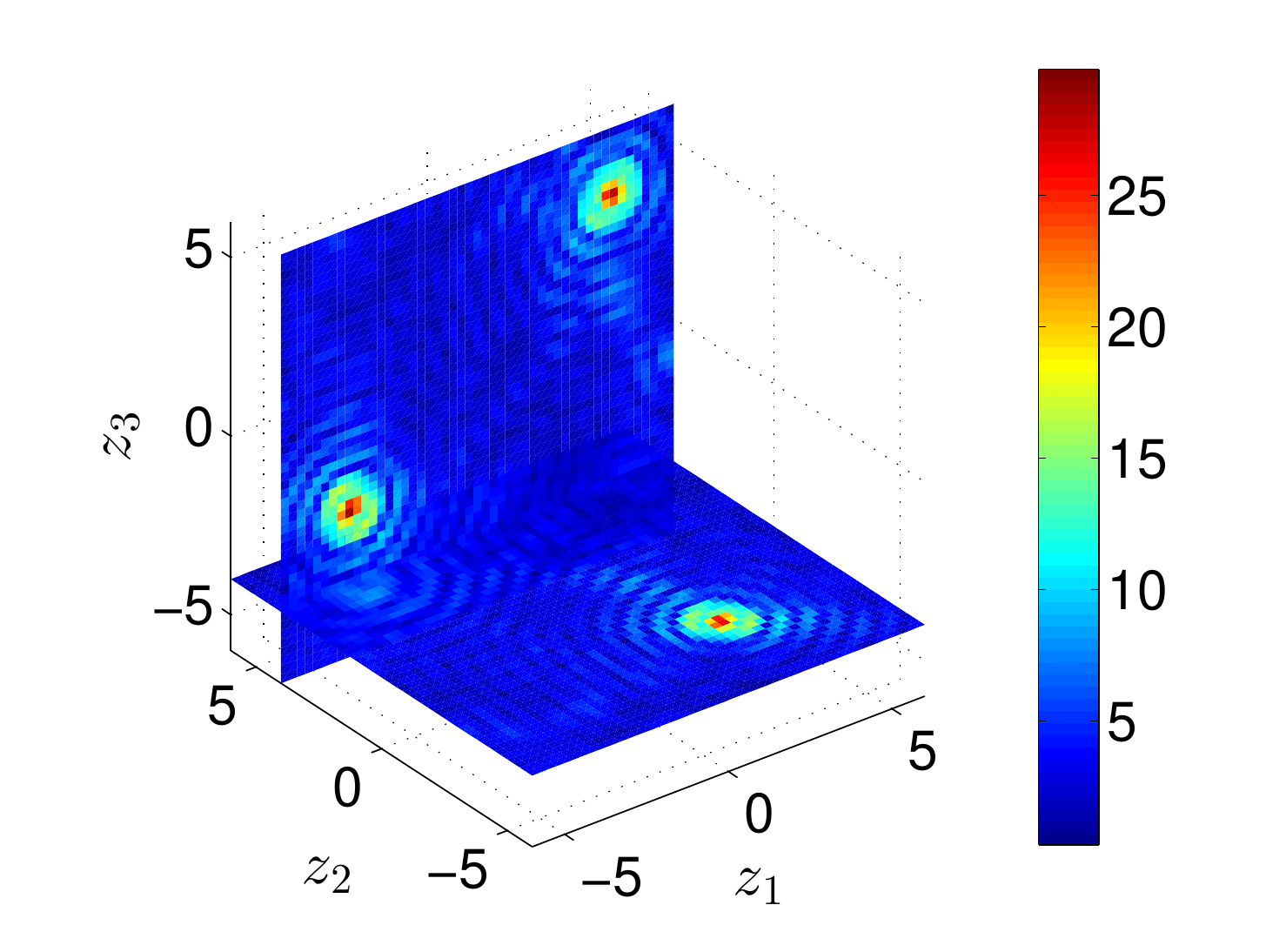}} \\
    \subfigure[]{\includegraphics[width=0.48\textwidth]{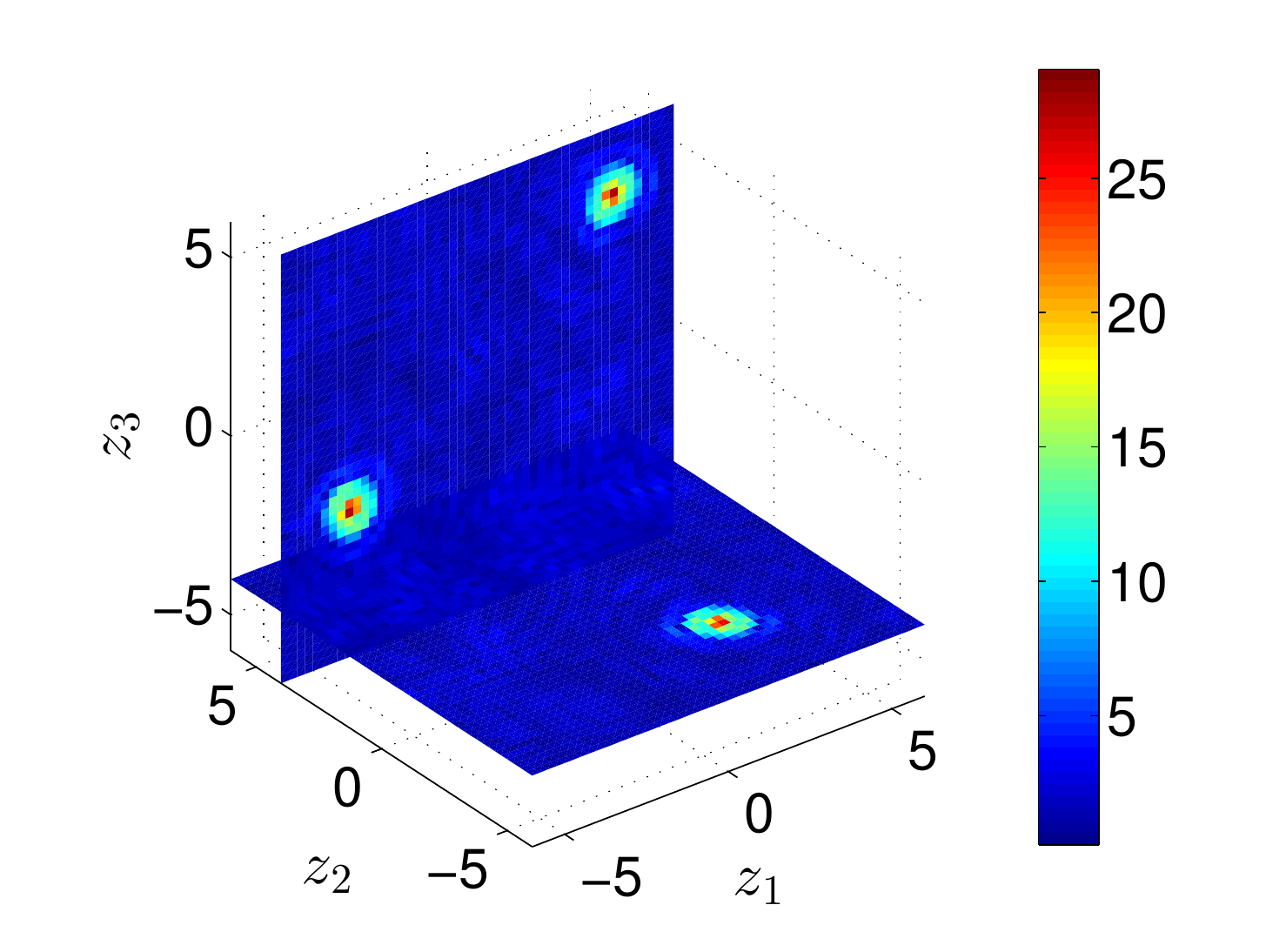}}
    \subfigure[]{\includegraphics[width=0.48\textwidth]{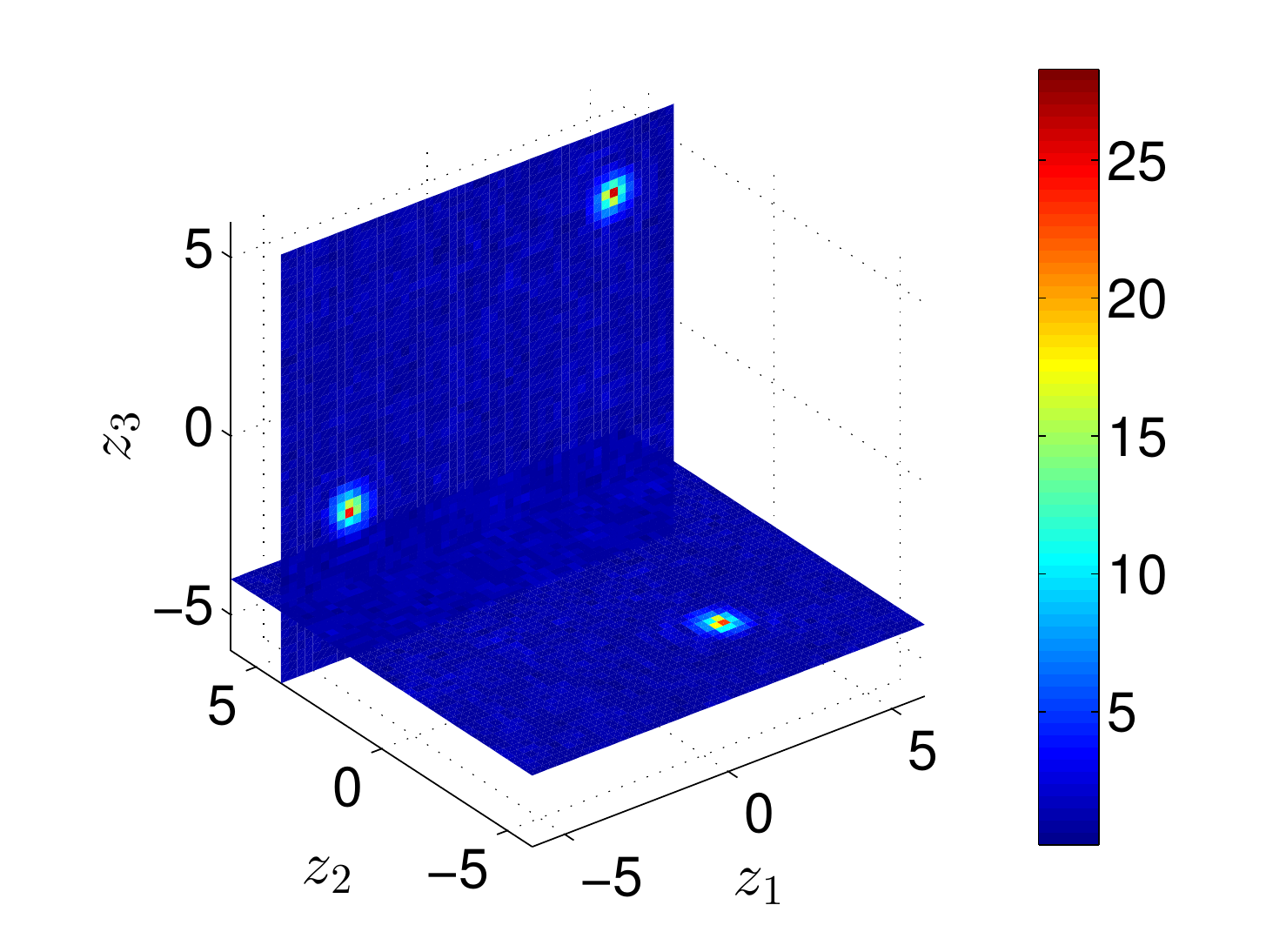}} \\
    \caption{Slice plots of the indicator function   $\sum_{\ell,\hbar=1}^3|\mathcal{I}_{\ell,\hbar}^{N,3,\epsilon}(z)|^2$  with different number of the frequencies $N$. (a) Exact locations marked by the small red points, (b) $N=1$, (c) $N=3$, (d) $N=7$.} \label{fig:Location_3D}
\end{figure}


\section*{Acknowledgment}

The work of Y. Guo  was supported by the NSFC grant under No. 11971133 and the Fundamental Research Funds for the Central Universities. The work of X. Wang was  supported by the Hong Kong Scholars Program grant under No. XJ2019005 and the NSFC grant under No. 12001140.


\end{document}